\documentclass{IEEEtran}
\usepackage{cite}
\usepackage{amsmath,amssymb,amsfonts}
\usepackage{algorithmic}
\usepackage{algorithm}
\usepackage{array}
\usepackage{color}
\usepackage[caption=false,font=normalsize,labelfont=sf,textfont=sf]{subfig}
\usepackage{textcomp}
\usepackage{amsmath}  
\usepackage{amsthm}   %
\usepackage{graphicx} 
\usepackage{epsfig}   
\usepackage{stfloats}
\usepackage{url}
\usepackage{verbatim}
\usepackage{lmodern} 
\usepackage{graphicx}
\usepackage{cite}

\newtheorem{lemma}{Lemma}  
\newtheorem{remark}{Remark}  
\newtheorem{theorem}{Theorem}

\allowdisplaybreaks
\def\BibTeX{{\rm B\kern-.05em{\sc i\kern-.025em b}\kern-.08em
    T\kern-.1667em\lower.7ex\hbox{E}\kern-.125emX}}
\begin{document}
\title{Distributed Circumferential Coverage Control in Non-Convex Annulus Environments}

\author{Chao Zhai~\thanks{Chao Zhai is with School of Automation, China University of Geosciences, Wuhan 430074, China, and with Hubei Key Laboratory of Advanced Control and Intelligent Automation for Complex Systems and also with Engineering Research Center of Intelligent Technology for Geo-Exploration, Ministry of Education. (e-mail: zhaichao@amss.ac.cn). The project was supported by the ``CUG Scholar" Scientific Research Funds at China University of Geosciences (Wuhan) (Project No.2020138)}}

\maketitle

\begin{abstract}
It has long been a prominent challenge in multi-agent systems to achieve distributed coverage of non-convex annulus environments while ensuring workload equalization among agents. To address this challenge, a distributed circumferential coverage control formulation is developed in this note by constructing a Riemannian metric for the navigation in the non-convex subregion while avoiding collisions with the region boundary. In addition, a distributed partition law is designed to balance the workload on the entire coverage region by endowing each agent with a virtual partition bar that slides along the inner boundary of coverage region. Theoretical analysis is conducted to ensure the exponential convergence of workload partition and asymptotic convergence of each agent towards the local optimum in its subregion. Finally, a case study is presented to demonstrate the effectiveness of the proposed coverage control approach.  
\end{abstract}


\section{Introduction}\label{sec:introduction}
As a key coordination strategy to deploy a group of agents in the concerned environment, multi-agent coverage control has found wide applications in various fields such as disaster management~\cite{ah24}, structural health monitoring~\cite{di24}, border security~\cite{zre21} and missile interception~\cite{zhai16}. In practice, environmental uncertainties and non-convexity of coverage region make it a challenging task for the distributed implementation of coverage missions~\cite{zhai21}.

Region partition provides a powerful framework to resolve  the distributed coverage issue of uncertain environment and efficient task allocation with the aid of divide-and-conquer scheme, 
as each agent only needs to fulfill the coverage task within its subregion. As a classic partition strategy, Voronoi partition enables to divide the entire coverage region into multiple subregions in a distributed manner~\cite{cor04}, with the benefits of linking its local optima to the subregion centroid for a certain coverage performance index. A plenty of its variants have also been proposed to allow for unknown environment~\cite{man24}, nonlinear dynamics of robots~\cite{enr08}, limited sensing capability~\cite{prat22}, and safety certificates~\cite{wu25}. 
Note that workload equalization is of great significance as it represents a balanced workload distribution among agents, which enables multi-agent system to efficiently complete the collective task. 
By introducing the generalized Voronoi partition, \cite{cor10} formulates this type of multi-agent coverage as area-constrained locational optimization problem and designs a Jacobi algorithm to find the weight assignment that satisfies the area constraints.
Nevertheless, they fail to achieve multi-agent coverage in non-convex region while simultaneously equalizing the workload among agents. 
As a result, a sectorial coverage control formulation is developed to address the coverage problem of star-shaped non-convex region by introducing the rotary pointers~\cite{zhai23}, which guarantees the exponential convergence of workload partition. 
The above coverage control strategy can be successfully extended to 3D non-convex environments~\cite{zau25}. In addition, \cite{zou25} proposes a hierarchical coverage control approach for non-holonomic agents to deal with complex poriferous environments by integrating Voronoi partition with pointer partition, where the Voronoi partition divides the whole region into multiple subregions based on the holes or obstacles, and the pointer partition further equalizes the workload on the subregion for each agent.

However, existing studies rely on a common reference point to implement the sectorial coverage control in star-shaped region, which inevitably restricts the distributed implementation of multi-agent coverage algorithm. In addition, the local control law for steering the agent towards the target is only applied to the convex subregion, which is infeasible in most practical scenarios. For the aforementioned reason, this note aims to develop a distributed coverage control formulation of multi-agent systems for non-convex annulus region with the guarantee of workload balance among subregions. 
In brief, the core contributions of this work are summarized as follows.
\begin{enumerate}
\item Propose a distributed coverage formulation for multi-agent deployment on non-convex annulus environment, which allows to deal with non-star-shaped non-convex region without predetermining a common reference point. 
\item Design a distributed partition law by sliding the partition bar along the inner boundary of coverage region, which enables the fully distributed implementation of coverage control strategy.
\item Construct a Riemannian metric for the coverage region, which allows to navigate the agent in the non-convex subregion while avoiding the collision with the region boundary.  
\end{enumerate}

The remainder of this note is outlined as follows. Section~\ref{sec:prob} formulates the coverage control problem for non-convex annulus environment. Section~\ref{sec:tech} provides technical analysis of proposed distributed coverage control algorithm, followed by a case study in Section~\ref{sec:case}. Finally, we draw a conclusion and discuss the future work in Section~\ref{sec:con}.  

{\sl Notation:} Throughout this paper, $\|\cdot\|$ indicates the Euclidean norm, and $\|\cdot\|_g$ denotes the norm induced by the Riemannian metric $g$.

\section{Problem Formulation}\label{sec:prob}
This section formulates the distributed coverage control problem of multi-agent systems in annulus non-convex non-star shaped environments. Consider a team of $N$ mobile agents in the coverage region $\Omega\subset\mathbb{R} ^2$, each governed by single-integrator dynamics
\begin{equation}\label{model}
\dot{\mathbf{p}}_i=\mathbf{u}_i, \quad i\in \mathbb{I}_N,
\end{equation}
where $\mathbf{p}_i\in\mathbb{R}^2$ denotes the position of the $i$-th agent in its local frame, and $\mathbf{u}_i\in\mathbb{R}^2$  is its control input with $\mathbb{I}_N=\{1,2,...,n\}$. 
To describe the monitoring area, we consider a two-dimensional annular region $\Omega$, bounded by two closed smooth boundary curves. 
The inner boundary curve  forms a convex region, while the outer boundary curve encloses 
a star-shaped region~(see Fig.~\ref{ske}). 
Note that the region can be described by $\Omega=\{q\in R^2: h(\mathbf{q})\geq0\}$ with a continuously differentiable function $h(\mathbf{q})$ and $h(\mathbf{q})=0$ if and only if $\mathbf{q}\in\partial\Omega$.
\begin{figure}[t!]
\centering\includegraphics[width=1.6in]{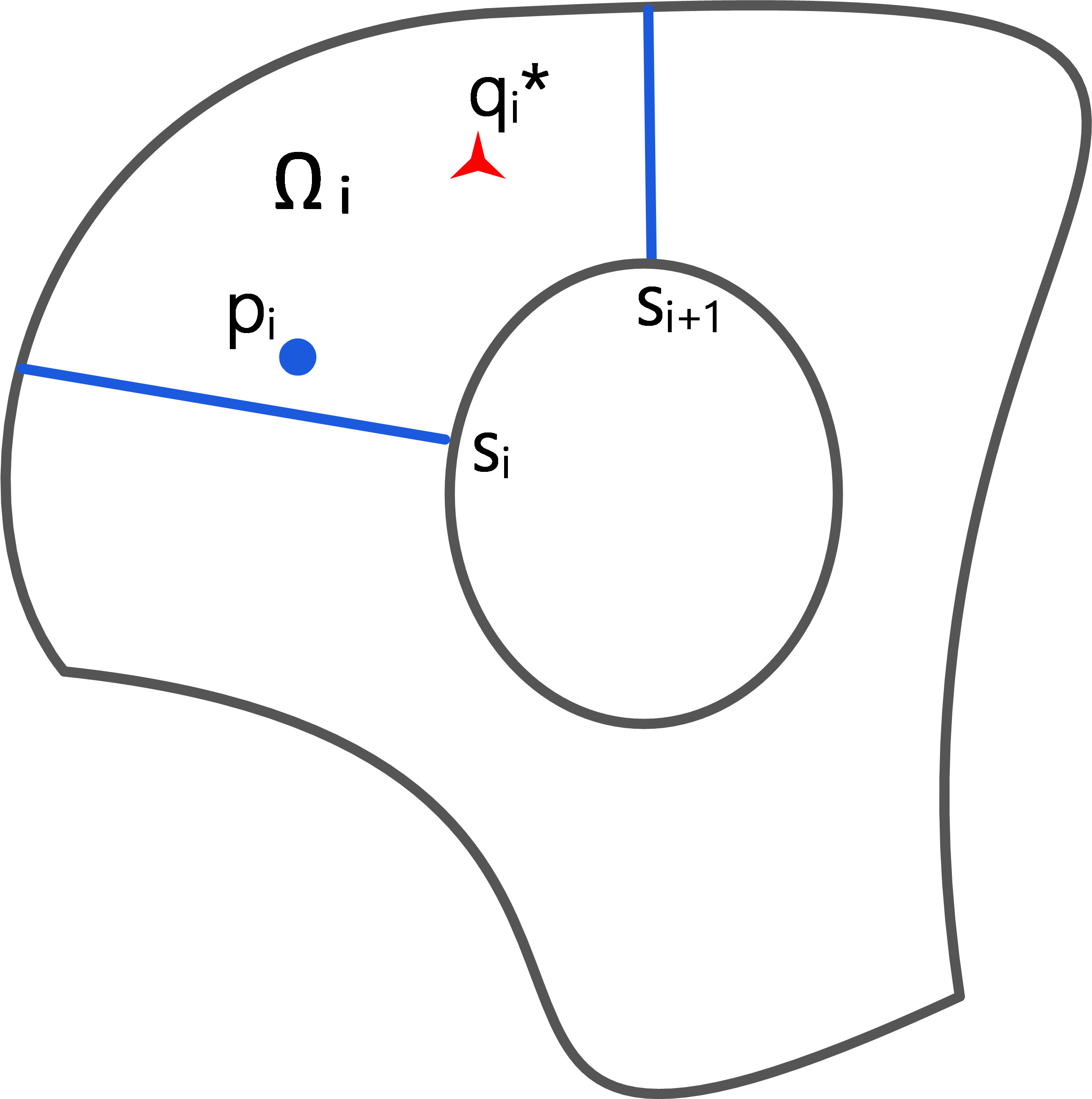}
\caption{\label{ske}The nonconvex annulus coverage region for multi-agent system. The blue lines denote the virtual partition bars, and the blue solid ball stands for the position of agent. The red triangle represents the optimal position for minimizing local coverage cost.}
\end{figure}
Each agent is equipped with a virtual partition bar, which enables to divide the whole region into multiple subregions and assign each agent with one subregion.
The virtual partition bar of the $i$-th agent is described by
\begin{equation}\label{bar}
\xi_i=\{\mathbf{q}\in\Omega: (\mathbf{q}-\mathbf{s}_i)^T\nu_i=0 \}, \quad i \in \mathbb{I}_N,
\end{equation}
where ${\nu}_i$ denotes the unit tangent vector to the inner boundary of coverage region at $\mathbf{s}_i$, and the workload on the \(i\)-th subregion $m_i$ increases along its direction.
The subregion assigned to the \(i\)-th agent is denoted 
by $\Omega_i$, which is bounded by  \(l_{\text{in}}\), 
\(l_{\text{out}}\), and two partition bars \(\xi_i\) and \(\xi_{i+1}\).
A workload density function \(\rho: \Omega\rightarrow \mathbb{R}^+\), constrained within the bounds \([\underline{\rho}, \bar{\rho}]\), is defined to quantify the spatial distribution of event occurrence probabilities over the domain \(\Omega\). Thus, the workload on the \(i\)-th subregion can be obtained by
\begin{equation}\label{mi}
m_i = \int_{\Omega_i} \rho(\mathbf{q}) d\mathbf{q},  \quad i\in \mathbb{I}_N.
\end{equation} 
In order to ensure workload balance among subregions, 
the partition dynamics of the $i$-th virtual partition bar is designed as follows 
\begin{equation}\label{part}
\dot{\mathbf{s}}_i = -\kappa_s(m_i-m_{i-1})\nu_i, \quad i\in \mathbb{I}_N 
\end{equation}
where \(\nu_i\in\mathbb{R}^2\) denotes the unit tangent vector to the inner boundary at \(\mathbf{s}_i\). Inspired by the work in~\cite{zhai23}, the performance index for multi-agent coverage 
is given by a smooth function
\begin{equation}\label{cost}
J(\bf{p},\mathbf{s}) =\sum\limits_{i=1}^N\int_{{\Omega_i}(\mathbf{s})}f(p_i,q)\rho(q)dq  
\end{equation}
with \(\mathbf{p}=(\mathbf{p}_1,\mathbf{p}_2,\ldots, \mathbf{p}_N)^T\) and \(\mathbf{s}=(\mathbf{s}_1, \mathbf{s}_2,\ldots,\mathbf{s}_N)^T\).
It is noteworthy that $f(\mathbf{p}_i,\mathbf{q})$ is used to quantify the cost of moving from $\mathbf{p}_i$ to $\mathbf{q}$ on the subregion \(\Omega_i\).
Essentially, this work aims to develop a distributed control algorithm for multi-agent system \eqref{model} to solve the following optimization problem
\begin{equation}
\begin{aligned}
& \min_{\mathbf{p},\mathbf{s}} 
J(\mathbf{p},\mathbf{s}) \\
& \text{s.t.} \quad m_i = m_j  \quad\forall i, j \in \mathbb{I}_N .
\end{aligned}
\end{equation}

\section{Technical Analysis}\label{sec:tech}
This section provides theoretical results of distributed coverage control algorithm for non-convex annulus environments. First of all, some key lemmas are presented below.

\subsection{Key lemmas}

\begin{lemma}
$$
\frac{\partial m_i}{\partial\mathbf{s}_i}\nu_i=\int_{\partial\Omega_{i-1}\cap\partial\Omega_i}\rho(\mathbf{q})d\mathbf{q},\quad \forall i\in \mathbb{I}_N.
$$
\end{lemma}

\begin{proof}
According to the divergence
theorem in differential geometry~\cite{cha84}, the partial time derivative of ${m}_i$ with respect to $\mathbf{s}_i$ is given by
\begin{align*}
\frac{\partial m_i}{\partial \mathbf{s}_i}= \int_{\partial \Omega_{i-1} \cap \partial \Omega_i} \rho(q) \mathbf{n}_i^T \frac{\partial \tau_i}{\partial \mathbf{s}_i} \, d\tau_i, \notag 
\end{align*}
where \(\partial \Omega_i\) denotes the boundary of the region \(\Omega_i\), and 
\(\mathbf{n}_i\)  denotes the unit outward normal vectors on \(\partial \Omega_i\). Thus, this leads to
\begin{equation*}
\begin{split}
\frac{\partial m_i}{\partial\mathbf{s}_i}\nu_i
&=\int_{\partial\Omega_{i-1} \cap \partial\Omega_i}\rho(q)\mathbf{n}_i^T \frac{\partial \tau_i}{\partial\mathbf{s}_i} d\tau_i \cdot \nu_i \\
&=\int_{\partial\Omega_{i-1} \cap \partial\Omega_i}\rho(q)\mathbf{n}_i^T \frac{\partial \tau_i}{\partial\mathbf{s}_i}\nu_i d\tau_i. 
\end{split}
\end{equation*}
It follows from $\tau_i\in\xi_i$ that 
$\tau_i^T \nu_i=(\mathbf{s}_i)^T\nu_i$, which leads to
$$
\frac{\partial \tau_i}{\partial\mathbf{s}_i}\nu_i=\left(\frac{\partial \tau_i^T\nu_i}{\partial\mathbf{s}_i}\right)^T=\left(\frac{\partial (\nu_i)^T\mathbf{s}_i}{\partial\mathbf{s}_i}\right)^T=\nu_i.
$$
Since both $\mathbf{n}_i$ and $\nu_i$ are unit vectors with the same direction, we get
\begin{equation*}
\begin{split}
\frac{\partial m_i}{\partial\mathbf{s}_i}\nu_i
&=\int_{\partial\Omega_{i-1} \cap \partial\Omega_i}\rho(q)\mathbf{n}_i^T\nu_idq  \\
&=\int_{\partial\Omega_{i-1} \cap \partial\Omega_i}\rho(q) dq. 
\end{split}
\end{equation*}
This completes the proof.
\end{proof}

\begin{lemma}\label{lem:part}
Partition dynamics~\eqref{part} ensures workload equalization among all subregions.
\end{lemma}
\begin{proof}
Construct the following Lyapunov function
\begin{equation*}
V = \frac{1}{2} \sum_{i=1}^{N} (m_i-\bar{m})^2.
\end{equation*}
The time derivative of Lyapunov function along~\eqref{part} is given by
\begin{align*}
\dot{V} &= \sum_{i=1}^{N} (m_i - \bar{m}) \dot{m}_i = \sum_{i=1}^{N} m_i \dot{m}_i - \bar{m} \sum_{i=1}^{N} \dot{m}_i = \sum_{i=1}^{N} m_i \dot{m}_i \\
&= \sum_{i=1}^{N} m_i \left(\frac{\partial m_i}{\partial s_i} \dot{s}_i-\frac{\partial m_{i+1}}{\partial s_{i+1}} \dot{s}_{i+1} \right) \\
&= \sum_{i=1}^{N} (m_i - m_{i-1}) \frac{\partial m_i}{\partial \mathbf{s}_i}\dot{\mathbf{s}}_i \\
&=-\kappa_s\sum_{i=1}^{N} (m_i - m_{i-1})^2\frac{\partial m_i}{\partial\mathbf{s}_i}\nu_i \\
&= -\kappa_s\sum_{i=1}^{N} (m_i-m_{i-1})^2
\int_{\partial\Omega_{i-1}\cap\partial\Omega_i}\rho(\mathbf{q})d\mathbf{q} \leq 0.
\end{align*}
It follows from $\dot{V} = 0$ that 
$(m_i-m_{i-1})^2=0$, $\forall~i\in \mathbb{I}_N$.
Therefore, the state of the system will approach the equilibrium point 
as time goes to infinity, which implies $m_1=m_2=\cdots= m_N$.
\end{proof}

\begin{lemma}\label{lemma3}
There exist positive constants  $c_1$ and $c_2$ such that 
$|m_i(t)-m_{i-1}(t)|\leq~c_1e^{-c_2t}$, $\forall t\geq 0$.
\end{lemma}

\begin{proof}
The time derivative of $V$ along \eqref{part} is given by 
\begin{align*}
\dot{V} &= -\kappa \sum_{i=1}^{N} (m_i - m_{i-1})^2 \int_{\partial\Omega_{i-1}\cap\partial\Omega_i}\rho(\mathbf{q})d\mathbf{q} \\
& \leq -\kappa\zeta\sum_{i=1}^{N} (m_i - m_{i-1})^2   
\end{align*}
with 
$$
\zeta=\inf_{\xi_i\in\Omega}\int_{\partial\Omega_{i-1}\cap\partial\Omega_i}\rho(\mathbf{q})d\mathbf{q}>0.
$$ 
It follows from Lemma 2 and Lemma~$3.2$ in~\cite{zhai23} that 
\begin{align*}
\dot{V} &\leq -\kappa \zeta\sum_{i=1}^{N} (m_i - m_{i-1})^2 \leq -\frac{2 \kappa \zeta \lambda_{\min}(\mathbf{S})}{N} V, 
\end{align*}
where $\mathbf{S}$ is a positive definite matrix~\cite{zhai23}. Solving the above differential inequality yields 
\begin{align*}
V(t) \leq V(0) e^{-\frac{2 \kappa \zeta \lambda_{\min}(\mathbf{S})}{N}t} 
\end{align*}
and 
$$
|m_i(t) - m_{i-1}(t)| \leq \sqrt{2V(t)} \leq c_1 e^{-c_2 t} 
$$
with \( c_1 = \sqrt{2V(0)} \) and \( c_2 = \kappa\zeta \lambda_{\min}(\mathbf{S})/{N} \). The proof is thus completed.
\end{proof}

\begin{lemma}\label{lem:conv}
$$
\lim_{t\to \infty}\dot{\mathbf{s}}_i(t)=\mathbf{0},~~\lim_{t\to\infty}\mathbf{s}_i(t)=\mathbf{s}_i^{*},~~\forall~i\in \mathbb{I}_N.
$$
\end{lemma}
\begin{proof}
It follows from Lemma~\ref{lem:part} and Eq.~\eqref{part} that 
$$
\lim_{t\to \infty}\dot{\mathbf{s}}_i(t)=\lim_{t\to \infty} -\kappa_s(m_i-m_{i-1})\nu_i=\mathbf{0}.
$$ 
In light of Lemma~\ref{lemma3}, we have 
$$
 \|\dot{\mathbf{s}}_i\|=\kappa_s|m_i-m_{i-1}|\leq \kappa_s c_1 e^{-c_2 t}.
$$
Let $\mathbf{s}_i=[{s}_{i,1},{s}_{i,2}]^T$, and for $k\in\{1,2\}$, we get 
\begin{align*}
\lim_{t \to \infty} s_{i,k}(t)-s_{i,k}(0) 
&= \int_0^{+\infty} \dot{s}_{i,k}(t) dt \\
&\leq \int_0^{+\infty} \|\dot{\mathbf{s}}_i(t)\| dt \\
&\leq \kappa_s  c_1 \int_0^{+\infty} e^{-c_2 t} dt = \frac{\kappa_s c_1}{c_2} < +\infty.
\end{align*}
According to Theorem 10.33 in~\cite{apo74}, 
$\dot{s}_{i,k}(t)$, $k\in\{1,2\}$ is improper Riemann-integrable on \([0, +\infty)\), and thus the improper integral 
$\int_0^{+\infty}\dot{s}_{i,k}(t) dt$
exists, which implies the existence of 
$\lim_{t \to \infty}{s}_{i,k}(t)$
due to the following equality 
$$
\lim_{t \to \infty} {s}_{i,k}(t) = {s}_{i,k}(0) + \int_0^{+\infty} \dot{s}_{i,k}(t) dt.
$$
Therefore, for any \(i \in \mathbb{I}_N\), \(\mathbf{s}_i(t)\) converges to the constant vector \(s_i^*\) as time goes to the infinity.
\end{proof}
Let $\mathbf{q}_i^{*}$ denote the position in the $i$-th sunregion that minimizes the local coverage cost in the local frame of the $i$-th agent. As a result, one has
$\mathbf{q}_i^{*}=\arg\inf_{\mathbf{q}\in \Pi_i}J_i$ with 
$$
J_i=\int_{\Omega_i(s)}f(\mathbf{p}_i,\mathbf{q}),\rho(\mathbf{q})d\mathbf{q}
$$ 
and 
$$
\Pi_i=\left\{\mathbf{p}_i\in\Omega_i|\nabla_{p_i}J=0\right\}.
$$
Since $\mathbf{q}_i^{*}$ is the solution to $\nabla_{p_i}J=\mathbf{0}$, it is an interior point in the subregion $\Omega_i$.
\begin{lemma}
If $\mathbf{H}_{J,p_i}(\mathbf{q}^{*}_i)$ is fully ranked, the following two claims hold. 
\begin{itemize}
\item[1)] $\lim_{t\to+\infty}\mathbf{\dot{q}}_i^{*}(t)=0$, $\forall{i}\in \mathbb{I}_N$.
\item[2)] $\mathbf{\dot{q}}_i^{*}(t)$ is absolutely integrable over time. 
\end{itemize}
\end{lemma}

\begin{proof}
For Claim 1), since $\mathbf{q}_i^{*}$ is the solution to the equation $\nabla_{p_i}J=0$, one gets
\begin{equation*}
\begin{split}
\frac{d\nabla_{p_i}J}{dt}
&=\mathbf{H}_{J,p_i}(\mathbf{q}^{*}_i)\mathbf{\dot{q}}_i^{*}+\frac{\partial\nabla_{p_i}J}{\partial{s}_{i}}\mathbf{\dot{s}}_{i}\\
&+\frac{\partial\nabla_{p_i}J}{\partial{s}_{i+1}}\mathbf{\dot{s}}_{i+1}=0.
\end{split}    
\end{equation*}
It follows from Lemma~\ref{lem:conv} that
$$
\lim_{t\to+\infty}\mathbf{\dot{s}}_{i}(t)=\lim_{t\to+\infty}\mathbf{\dot{s}}_{i+1}(t)=0.
$$
Note that both ${\partial\nabla_{p_i}J}/{\partial{s}_{i}}$ and ${\partial\nabla_{p_i}J}/{\partial{s}_{i+1}}$ are bounded in the coverage region, which leads to 
$$
\lim_{t\to+\infty}\mathbf{H}_{J,p_i}(\mathbf{q}^{*}_i)\mathbf{\dot{q}}_i^{*}(t)=0.
$$
Considering that $\mathbf{H}_{J,p_i}(\mathbf{q}^{*}_i)$ is fully ranked, one gets  
$\lim_{t\to+\infty}\mathbf{\dot{q}}_i^{*}(t)=0$.
For Claim 2), since $\mathbf{H}_{J,p_i}(\mathbf{q}^{*}_i)$ is fully ranked, one obtains
$$
\mathbf{\dot{q}}_i^{*}=-\mathbf{H}^{-1}_{J,p_i}\left(\frac{\partial\nabla_{p_i}J}{\partial{s}_{i}}\mathbf{\dot{s}}_{i}
+\frac{\partial\nabla_{p_i}J}{\partial{s}_{i+1}}\mathbf{\dot{s}}_{i+1}\right),
$$
which leads to
\begin{equation*}
\begin{split}
\|\mathbf{\dot{q}}_i^{*}\|
&\leq \|\mathbf{H}^{-1}_{J,p_i}\|\cdot\left\|\frac{\partial\nabla_{p_i}J}{\partial{s}_{i}}\mathbf{\dot{s}}_{i}
+\frac{\partial\nabla_{p_i}J}{\partial{s}_{i+1}}\mathbf{\dot{s}}_{i+1}\right\| \\  
&\leq \|\mathbf{H}^{-1}_{J,p_i}\|\cdot\left\|\frac{\partial\nabla_{p_i}J}{\partial{s}_{i}}\right\|\cdot\|\mathbf{\dot{s}}_{i}\|\\
&+\|\mathbf{H}^{-1}_{J,p_i}\|\cdot\left\|\frac{\partial\nabla_{p_i}J}{\partial{s}_{i+1}}\right\|\cdot\|\mathbf{\dot{s}}_{i+1}\| \\ 
&\leq K_HK_i\|\mathbf{\dot{s}}_{i}\|+K_HK_{i+1}\|\mathbf{\dot{s}}_{i+1}\|
\end{split}    
\end{equation*}
because of
$$
\|\mathbf{H}^{-1}_{J,p_i}\|\leq K_H, \quad \left\|\frac{\partial\nabla_{p_i}J}{\partial{s}_{j}}\right\|\leq K_j,~~j\in\{i,i+1\}.
$$
As a result, one has
\begin{equation*}
\begin{split}
\int_{0}^{+\infty}\|\mathbf{\dot{q}}_i^{*}(t)\|dt
&\leq K_HK_i\int_{0}^{+\infty} \|\mathbf{\dot{s}}_{i}(t)\|dt \\
&+K_HK_{i+1}\int_{0}^{+\infty} \|\mathbf{\dot{s}}_{i+1}(t)\|dt \\
&\leq K_H\left(\frac{\kappa c_{1,i}K_i}{c_{2,i}}+\frac{\kappa c_{1,i+1}K_{i+1}}{c_{2,i+1}}\right)<+\infty,
\end{split}    
\end{equation*}
which completes the proof.
\end{proof}

Construct an energy-like function as follows 
$$
E(t)=\frac{1}{2}d^2_g(\mathbf{p}_i(t),\mathbf{q}^{*}_i(t)),
$$
where $d_g(\mathbf{p}_i(t),\mathbf{q}^{*}_i(t))$ denotes the Riemannian distance between $\mathbf{p}_i(t)$ and $\mathbf{q}^{*}_i(t)$ in the coverage region. 

\begin{lemma}\label{Elim}
$E(t)$ converges to a finite limit as time goes to the infinity.
\end{lemma}
\begin{proof}
The time derivative of $E(t)$ along the trajectory of \eqref{model} and the dynamics of the optimal position is 
given by
\begin{equation*}
\begin{split}
\dot{E}(t)&=\frac{\partial E}{\partial \mathbf{p}_i}\mathbf{\dot{p}}_i
+\frac{\partial E}{\partial \mathbf{q}_i^{*}}\mathbf{\dot{q}}_i^{*} \\
&=-\kappa_p\left\|\frac{\partial E}{\partial \mathbf{p}_i}\right\|_g^2
+\xi(t)
\end{split}    
\end{equation*}
with
$\xi(t)=\langle{\partial E}/{\partial \mathbf{q}_i^{*}},\mathbf{\dot{q}}_i^{*}\rangle_g$. It follows from $\lim_{t\to+\infty}\mathbf{\dot{q}}_i^{*}(t)=0$ that $\lim_{t\to+\infty}\xi(t)=0$. Moreover, one has
\begin{equation*}
\begin{split}
\int_{0}^{+\infty}\xi(t)dt
&\leq\int_{0}^{+\infty}|\xi(t)|dt \\
&\leq \int_{0}^{+\infty}\left\|\frac{\partial E}{\partial \mathbf{q}_i^{*}}\right\|_g \cdot\|\mathbf{\dot{q}}_i^{*}(t)\|_gdt \\
&\leq K_q \int_{0}^{+\infty}\|\mathbf{\dot{q}}_i^{*}(t)\|_gdt
\end{split}    
\end{equation*}
with the positive constant $K_q$. In light of Lemma $6$, one gets
$$
\int_{0}^{+\infty}\xi(t)dt\leq K_q \int_{0}^{+\infty}\|\mathbf{\dot{q}}_i^{*}(t)\|_gdt<+\infty,
$$
which indicate the existence of the integral $\int_{0}^{+\infty}\xi(t)dt$.
For the convenience of analysis, introduce a variable as follows
$$
W(t)=E(t)-\int_{0}^{t}\xi(\tau)d\tau.
$$
Considering that $E(t)\geq0$ and there exists a positive constant $K_{\xi}$ such that
$$
\int_{0}^{t}\xi(\tau)d\tau\leq \int_{0}^{+\infty}|\xi(\tau)|d\tau\leq K_{\xi},
$$
one has $W(t)\geq-K_{\xi}$. Moreover, the time derivative of $W(t)$ is given by 
\begin{equation*}
\begin{split}
\dot{W}(t)=\dot{E}(t)-\xi(t)=-\kappa_p\left\|\frac{\partial E}{\partial \mathbf{p}_i}\right\|_g^2\leq0.
\end{split}    
\end{equation*}
By the Monotone Convergence Theorem, $W(t)$ converges to a finite limit as time goes to the infinity. This leads to the existence of $\lim_{t\to+\infty}E(t)$ due to the existence of the integral $\int_{0}^{+\infty}\xi(t)dt$, which completes the proof.
\end{proof}

\begin{lemma}\label{Euni}
$\dot{E}(t)$ is uniformly continuous with respect to time.
\end{lemma}
\begin{proof}
For the uniform continuity of $\dot{E}(t)$, it suffices to demonstrate the boundedness of $\ddot{E}(t)$. Since $\dot{E}(t)$ is given by
$$
\dot{E}(t)=\frac{\partial E}{\partial \mathbf{p}_i}\mathbf{\dot{p}}_i
+\frac{\partial E}{\partial \mathbf{q}_i^{*}}\mathbf{\dot{q}}_i^{*},
$$
one gets
\begin{equation}\label{ddE}
\begin{split}
\ddot{E}(t)&=\left(\frac{\partial^2 E}{\partial \mathbf{p}_i\partial \mathbf{p}_i}\mathbf{\dot{p}}_i+\frac{\partial^2 E}{\partial \mathbf{p}_i\partial \mathbf{q}_i^{*}}\mathbf{\dot{q}}_i^{*}
\right)^T\mathbf{\dot{p}}_i\\
&+\left(\frac{\partial^2 E}{\partial \mathbf{q}_i^{*}\partial \mathbf{q}_i^{*}}\mathbf{\dot{q}}_i^{*}+\frac{\partial^2 E}{\partial \mathbf{q}_i^{*}\partial\mathbf{p}_i}\mathbf{\dot{p}}_i
\right)^T\mathbf{\dot{q}}_i^{*}\\
&+\frac{\partial E}{\partial\mathbf{p}_i}\mathbf{\ddot{p}}_i+\frac{\partial E}{\partial \mathbf{q}_i^{*}}\mathbf{\ddot{q}}_i^{*}. 
\end{split}    
\end{equation}
In addition, considering that
$$
\mathbf{\dot{p}}_i=-{\kappa}_p\frac{{\partial E}}{{\partial {\mathbf{p}_i}}},
$$
one obtains
$$
\mathbf{\ddot{p}}_i=-{\kappa}_p\left(\frac{\partial^2 E}{\partial \mathbf{p}_i\partial \mathbf{p}_i}\mathbf{\dot{p}}_i+\frac{\partial^2 E}{\partial \mathbf{p}_i\partial \mathbf{q}_i^{*}}\mathbf{\dot{q}}_i^{*}
\right).
$$
Similarly, 
$$
\mathbf{\dot{q}}_i^{*}=-\mathbf{H}^{-1}_{J,p_i}\left(\frac{\partial\nabla_{p_i}J}{\partial{s}_{i}}\mathbf{\dot{s}}_{i}
+\frac{\partial\nabla_{p_i}J}{\partial{s}_{i+1}}\mathbf{\dot{s}}_{i+1}\right)
$$
with
$\dot{\mathbf{s}}_j = -\kappa(m_j-m_{j-1})\nu_j$, 
$j\in\{i,i+1\}$. In light of Frenet–Serret formulas~\cite{pra22}, $\dot{\nu}_j$ can be computed by
$$
\dot{\nu}_j=\frac{d\nu_j}{dl}\cdot\frac{dl}{dt}=
\frac{\partial\nu_j}{\partial \mathbf{s}_j}\|\mathbf{\dot{s}}_j\|=\kappa_s(m_j-m_{j-1})\kappa(\mathbf{\dot{s}}_j)\mathbf{n}_j,
$$
where $l$ denotes the arc length, and $\kappa(\mathbf{\dot{s}}_j)$ and $\mathbf{n}_j$ represent the curvature of the inner boundary curve and the unit normal vector at the point $\mathbf{\dot{s}}_j$, respectively. Since $E(t)$ and $J(\mathbf{p},\mathbf{s})$ are both smooth functions with respect to their respective variables, $\mathbf{\dot{p}}_i$, $\mathbf{\dot{q}}_i^{*}$, and $\mathbf{\ddot{p}}_i$ are all bounded on the compact coverage region, which implies that the first three terms in \eqref{ddE} are bounded. For the last term in in \eqref{ddE}, because ${\partial{E}}/{\partial \mathbf{q}_i^{*}}$ is bounded, it suffices to examine the boundedness of $\mathbf{\ddot{q}}_i^{*}$. 
\end{proof}

\subsection{Design of distributed coverage controller}
\label{sec:design-controller}     
In the following, a distributed controller is designed for each agent to avoid the collision with region boundaries or obstacles and remain within its feasible coverage region.
Inspired by~\cite{bha14}, a Riemannian metric  
\(g\) is introduced on each subregion, which induces a length metric $d_g(\mathbf{x},\mathbf{y})$.
This metric guides each agent toward its optimal centroid while avoiding collisions with the boundary of the subregion.
Given the position of the $i$-th agent $\mathbf{p}_i$ and 
the target position $\mathbf{q}$, we can obtain a path 
\(\gamma: [0,1]\to\Omega_i\), where 
\(\gamma(t) =[x^1(t),x^2(t)]^T\in \mathbb{R}^2\), $\gamma(0)=\mathbf{p}_i$ and $\gamma(1)=\mathbf{q}$.  
Let \(L(\gamma)\) denote the path length of \(\gamma\), 
and it can be computed by
$$
L(\gamma)=\int_{0}^{1}\sqrt{\sum_{i,j=1}^{2}g_{ij}\frac{{d{{x}^i}}}{{dt}}\frac{{d{{x}^j}}}{{dt}}}dt,
$$
where the Riemannian metric $g$ is designed as
$g_{ij}=\delta_{ij}/h^2(\mathbf{q})$, $\forall i,j\in\{1,2\}$.
Here, the Kronecker delta $\delta_{ij}$ satisfies $\delta_{ij}=1$ for $i=j$ and $\delta_{ij}=0$, otherwise.
In particular, 
\(d_g(\mathbf{p}_i,\mathbf{q})\) is equal to the infimum of the lengths of all rectifiable paths connecting 
\(\mathbf{p}_i\) and $\mathbf{q}$, that is,
\begin{equation}\label{dl}
d_g(\mathbf{p}_i,\mathbf{q})=\inf_{\gamma\in \Omega} L(\gamma),
\end{equation}
which satisfies the fundamental metric properties. Then the control input for the $i$-th agent in the subregion \(\Omega_i\) is designed as 
\begin{equation}\label{Control_Input}
\mathbf{u}_i=-{\kappa}_p\frac{{\partial E}}{{\partial {\mathbf{p}_i}}},
\end{equation}
where \(\kappa_p\) is a positive gain. 

\begin{theorem}
Multi-agent system (\ref{model}) with partition dynamics (\ref{part}) and control input \eqref{Control_Input} ensures that each agent can reach the optimal configuration of its subregion without getting stuck by the boundaries of coverage region or obstacles. 
\label{Theorem38}
\end{theorem}
\begin{proof}
It follows from Lemma~\ref{Elim} and Lemma~\ref{Euni} that $E(t)$ converges to a finite limit as time goes to the infinity and $\dot{E}(t)$ is uniformly continuous. In light of Barbalat's lemma, one has $\lim_{t\to+\infty}\dot{E}(t)=0$, which leads to $\lim_{t\to+\infty}\|\nabla_{p_i}E\|_g^2=0$ due to $\lim_{t\to+\infty}\langle\mathbf{\nabla_{q_i}E,\dot{q}^{*}_i}\rangle_g=0$.
Since $\|\nabla_{p_i}E\|_g^2=d^2_g(\mathbf{p}_i(t),\mathbf{q}^{*}_i(t))$, one gets 
$$
\lim_{t\to+\infty}d_g(\mathbf{p}_i(t),\mathbf{q}^{*}_i(t))=0.
$$ 
According to the triangle inequality 
$$
0\leq d_g(\mathbf{p}_i(t),\mathbf{q}^{*}_i)\leq  d_g(\mathbf{p}_i(t),\mathbf{q}^{*}_i(t))+d_g(\mathbf{q}^{*}_i(t),\mathbf{q}^{*}_i)
$$
and $\lim_{t\to+\infty}\mathbf{q}^{*}_i(t)=\mathbf{q}^{*}_i$, one has 
$$
\lim_{t\to+\infty}d_g(\mathbf{p}_i(t),\mathbf{q}^{*}_i)=0,
$$ 
which indicates $\lim_{t\to+\infty}\mathbf{p}_i(t)=\mathbf{q}^{*}_i$.
The proof by contradiction is employed for the collision avoidance with region boundaries.  Assume the $i$-th agent collides with the region boundary $\partial\Omega$ at some time $t_c$, which implies $\mathbf{p}_i(t_c)\in \partial\Omega$ and $h(\mathbf{p}_i(t_c))=0$. According to the definition of Riemannian metric $g$, its components $g_{ij}$ goes to the infinity as the agent approaches the boundary (i.e., $\mathbf{p}_i(t)\to\mathbf{p}_i(t_c)$).
Then the path length of the $i$-th agent during the time interval $[0,t_c]$ is given by
$$
L=\int_{0}^{t_c} \|\mathbf{\dot{p}}_i(\tau)\|_gd\tau=\int_{0}^{t_c} \frac{\|\mathbf{\dot{p}}_i(\tau)\|_2}{h(\mathbf{p}_i)}d\tau.
$$
Because of 
$\dot{h}(\mathbf{p}_i)=\nabla{h}^T\dot{\mathbf{p}}_i$,
one gets
$$
|\dot{h}(\mathbf{p}_i)|=|\nabla{h}^T\dot{\mathbf{p}}_i|\leq\|\nabla{h}\|_2\cdot\|\dot{\mathbf{p}}_i\|_2.
$$
Since $h(\mathbf{q})$ is continuously differentiable on the closed region $\Omega$, there exists $c_h>0$ such that $\|\nabla{h}\|\leq c_h$,
which leads to
$$
-\dot{h}(\mathbf{p}_i(\tau)) \leq|\dot{h}(\mathbf{p}_i(\tau))|\leq c_h\|\mathbf{\dot{p}}_i(\tau)\|_2.
$$
Thus, it follows that
\begin{equation*}
\begin{split}
L&=\int_{0}^{t_c} \frac{\|\mathbf{\dot{p}}_i(\tau)\|_2}{h(\mathbf{p}_i)}d\tau \\
&\geq -\frac{1}{c_h}\int_{0}^{t_c} \frac{\dot{h}(\mathbf{p}_i(\tau))}{h(\mathbf{p}_i(\tau))} d\tau \\
&=\frac{1}{c_h}\int_{t_c}^{0}\frac{dh(\mathbf{p}_i(\tau))}{h(\mathbf{p}_i(\tau))}=\frac{1}{c_h}\ln h(\mathbf{p}_i(\tau))\Big|_{t_c}^{0}\longrightarrow+\infty.
\end{split}    
\end{equation*}
This indicates that the length of any path that touches the region boundary is infinite large. However, an agent with the finite velocity can never travel an infinite distance in a finite time $t_c$, which  contradicts with our assumption and indicates the false assumption.
\end{proof}

\begin{remark}
If $\mathbf{q}_i^*$ is on the boundary of coverage region, a control law should be redesigned to drive the $i$-th agent towards the closest position on the region boundary, and then it moves along the boundary to approach the target point $\mathbf{q}_i^*$.  
\end{remark}


\section{Case Study}\label{sec:case}

Case studies are carried out in this section to validate the proposed coverage algorithm through numerical simulations.
\begin{figure}[t!]
\center
\subfloat{\includegraphics[width=0.50\linewidth]{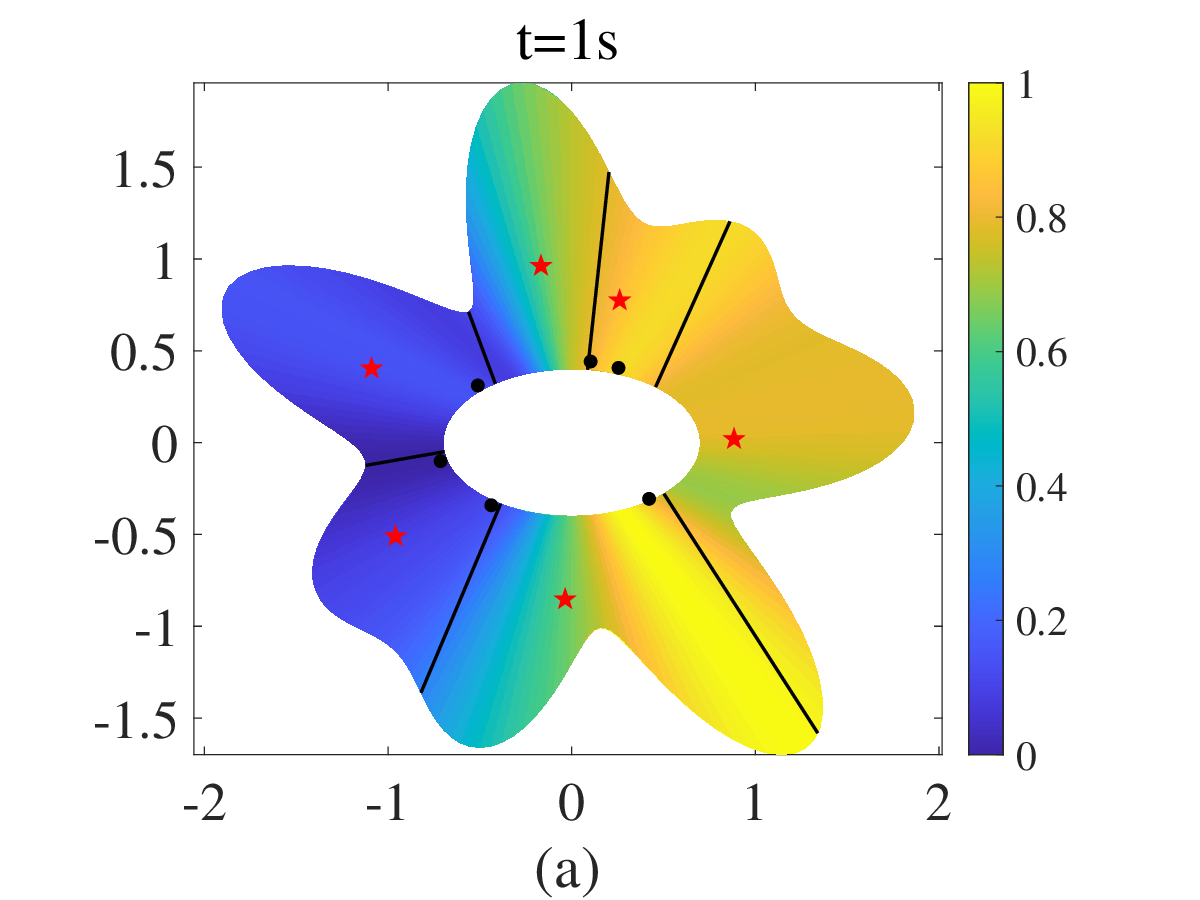}}
\subfloat{\includegraphics[width=0.50\linewidth]
{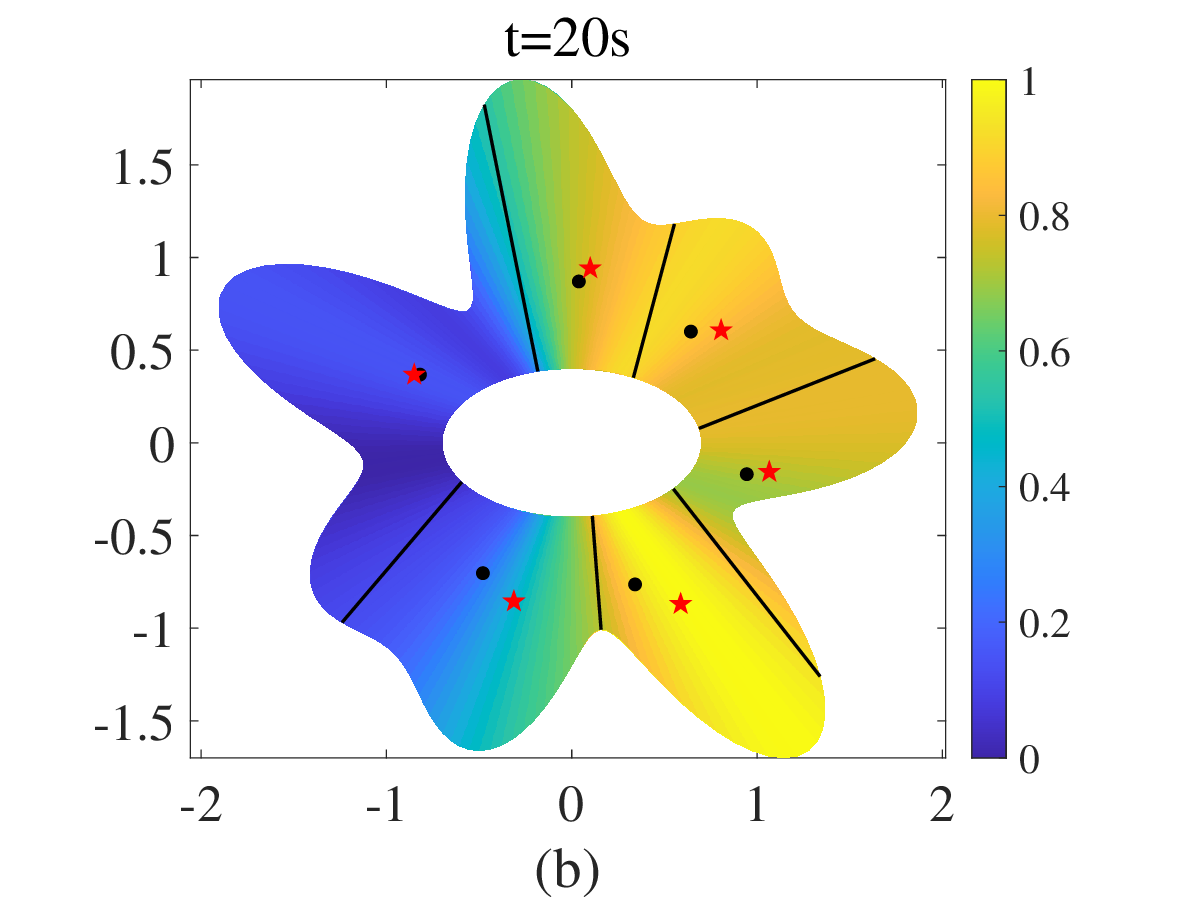}} \hfill
\subfloat{\includegraphics[width=0.50\linewidth]{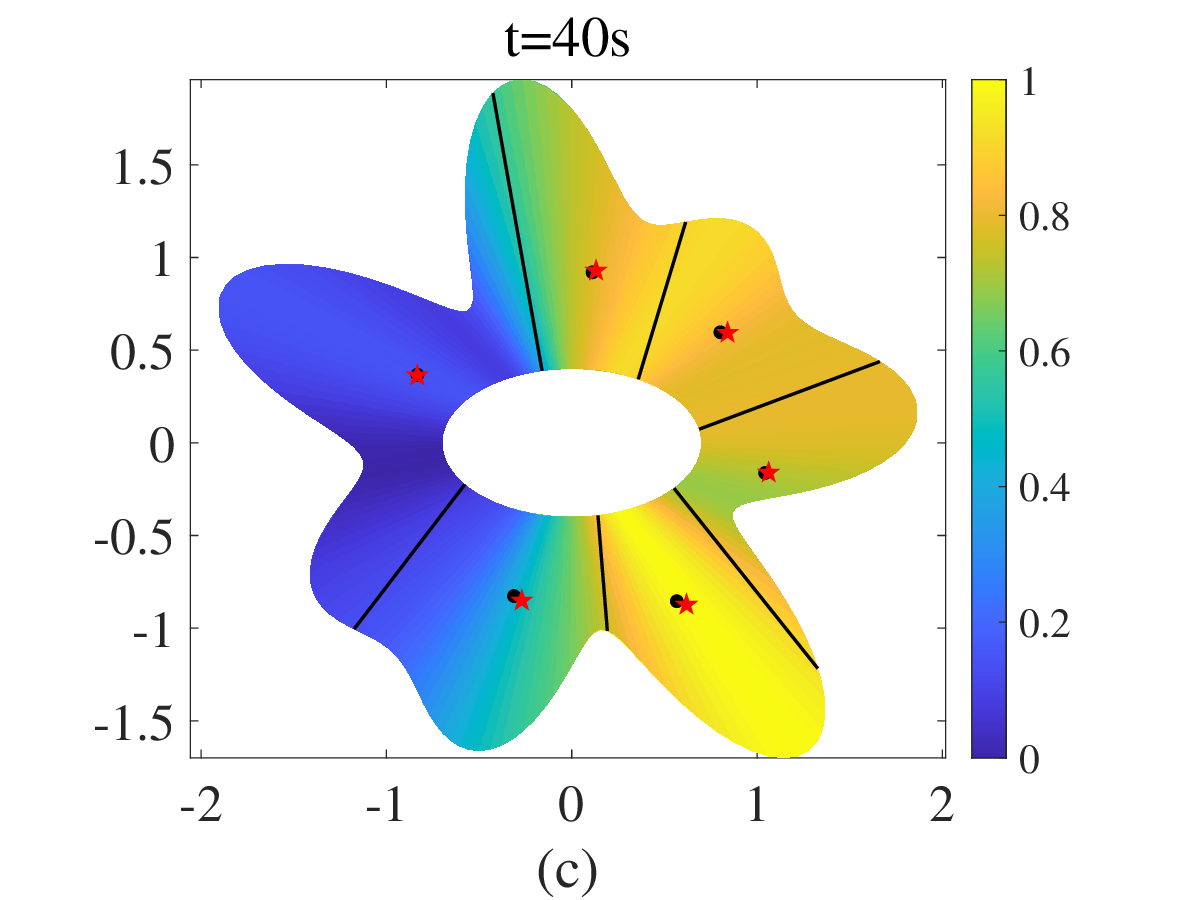}}
\subfloat{\includegraphics[width=0.50\linewidth]{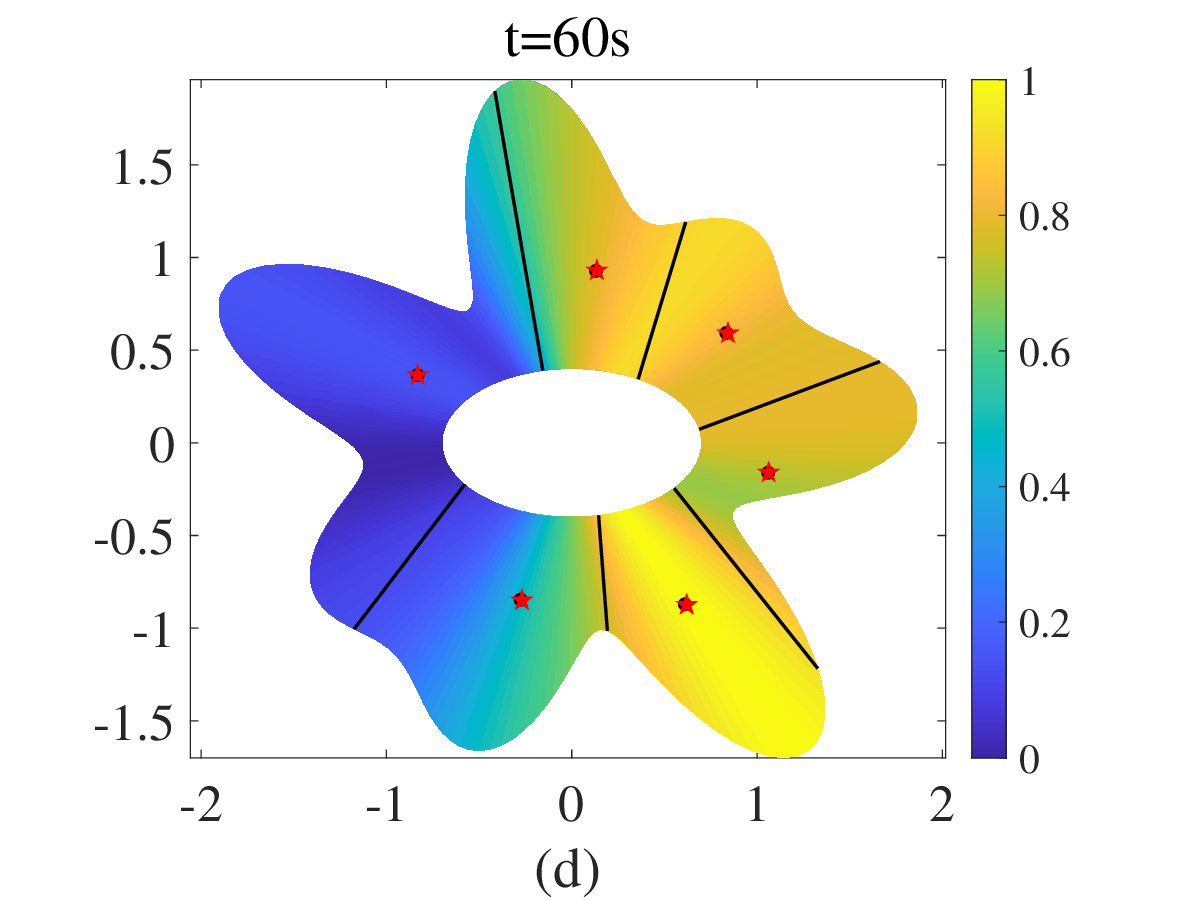}}\caption{\label{fig:main} Snapshots illustrating the simulation results of the MAS, showing the region partitions, subregion centroids, and agent positions. The mobile agents are marked by black points, the centroids of the subregions are represented by red stars.}
\end{figure}
In the simulation, a team of $6$ robots moves within a $4~\mathrm{m} \times 3.6~\mathrm{m}$ area. 
Due to the space limitation of the experimental testbed, two closed curves with band-shaped regions and internal holes are used. These curves are described by the polar equations 
$R_{\mathrm{in}}(\theta) = 1 / \sqrt{\cos^2\theta/0.49+\sin^2\theta/0.16}$ and $R_{\mathrm{out}}(\theta)  = 1.5+0.3\sin5\theta 
+0.3 \cos7\theta$. The workload density over the coverage area is described by the function
$\rho(r,\theta)=\text{exp}(\sin^2 \theta + \cos\theta)+0.01r$.
Note that the hole regions refer to obstacles that block the movement of the robots. 
The robots, following our improved control law, can effectively avoid collisions with the environment boundaries and obstacles, 
and successfully reach the target points of their respective subregions. 
Other parameters are chosen empirically as $\kappa_p = 0.1$ and $\kappa_s = 0.02$.
Each agent moves towards its corresponding optima while performing the rotary pointer partition in order to achieve an even distribution of the workload. 
The evolution of the optima is dependent  on the positions of partition bars. 
As shown in Fig.~\ref{fig:main},
the coverage region \( \Omega \) is nonconvex, where black points denote the mobile agents, and red stars represent the optima of subregions.
In the simulation, the initial positions of agents and virtual partition bars are randomly initialized.
At $t = 20\,\text{s}$, although the agents have not yet fully reached the optima of their respective subregions, 
the region partitioning is nearly complete. 
By $t=40\,\text{s}$, partition bars have converge, and the workload is accordingly partitioned. The agents are very close to the optima of their respective subregions. 
As shown in Fig.~\ref{fig:main}, each agent eventually arrives at the optima of its respective subregion, and the area of each subregion is determined by the workload density. The simulation results indicate that the proposed coverage control algorithm is stable and effective, and all agents successfully reach the optima of their respective subregions, thereby achieving a balanced workload partition while simultaneously optimizing the entire coverage performance.

\section{Conclusion}\label{sec:con}
This paper aimed to address the coverage control problem of multi-agent systems in non-convex annulus environment, and a distributed control algorithm was proposed to deploy the agent for optimizing the coverage performance while dividing the coverage region into subregions and achieving workload balance among subregions. Collision avoidance with region boundary was guaranteed by leveraging a Riemannian metric to navigate each agent in its subregion.  Future work may include the extension of the present study to dynamic uncertain environment with non-holonomic constraints.



\end{document}